\documentclass[letterpaper, 10 pt, conference]{ieeeconf} 
\IEEEoverridecommandlockouts

\usepackage{cite}
\usepackage[cmex10]{amsmath}
\usepackage{url}  
\usepackage{graphics} 
\usepackage{epsfig} 
\usepackage{amssymb}  
\usepackage{fancyhdr}
\usepackage{color}
\usepackage{algorithm}
\usepackage{algpseudocode}
\usepackage{bbm}

\newtheorem{thm}{Theorem}

\newtheorem{defn}{Definition}

\newtheorem{lem}{Lemma}

\newtheorem{prop}{Proposition}

\newenvironment{sketch-proof}{{\it Sketch of the Proof:\ }}{ \hfill}

\newcommand{\R}{{\mathbb R}}

\newcommand{\cC}{{\mathcal C}}
\newcommand{\cD}{{\mathcal D}}

\newcommand{\cF}{{\mathcal F}}

\newcommand{\cH}{{\mathcal H}}

\newcommand{\cL}{{\mathcal L}}
\newcommand{\cM}{{\mathcal M}}

\newcommand{\cP}{{\mathcal P}}

\newcommand{\cS}{{\mathcal S}}

\newcommand{\trace}{{\text{tr}}}

\newcommand{\xstar}{x^{\star}}
\newcommand{\Xstar}{X^{\star}}

\algrenewcommand\algorithmicrequire{\textbf{Initialize:}}
\algrenewcommand\algorithmicensure{\textbf{Return:}}

\DeclareMathOperator*{\argmin}{arg\,min}

\title{Distance to the Nearest Stable Metzler Matrix} 

\author{James Anderson
\thanks{JA is with the Department of Computing + Mathematical Sciences, California Institute of Technolgy, Pasadena, CA 91125, {\tt\small james@caltech.edu} This work was funded by  NSF CNS award 1545096, ECCS award 1619352, CCF award 1637598, and ARPA-E award GRID DATA.}}

\begin{document}

\maketitle

\begin{abstract}
This paper considers the non-convex problem of finding the nearest Metzler matrix to a given possibly unstable matrix. Linear systems whose state vector evolves according to a Metzler matrix have many desirable properties in analysis and control with regard to scalability. This motivates the question, how close (in the Frobenius norm of coefficients) to the nearest Metzler matrix are we?  Dropping the Metzler constraint, this problem has recently been studied using the theory of dissipative Hamiltonian (DH) systems, which provide a helpful characterization of the feasible set of stable matrices. This work uses the DH theory to provide a block coordinate descent consisting of a quadratic program with favourable structural properties and a semidefinite program for which recent diagonal dominance results can be used to improve tractability.

\end{abstract}

\IEEEpeerreviewmaketitle

\section{Introduction}\label{se:Intro}

\subsection{Introduction}
In this paper we consider the problem of finding the nearest stable Metzler matrix to a given non-Metzler (possibly unstable) matrix. Metzler matrices and the related class of positive systems have become very popular in the systems and control community as they can model a wide class of physically important systems, e.g. population vectors in ecological networks, transportation systems, chemical reaction networks, etc. Moreover, such systems have nice theoretical properties that lend themselves well to scalable computational analysis. Positive systems have been studied from a variety of perspectives, including stability and performance analysis \cite{Ran15}, robust synthesis~\cite{TanL11}, model reduction~\cite{SooA17}, and optimization~\cite{LavRL11}. 

Thus, it seems that given a matrix which is not Metzler, it may be worth the effort to determine whether it is `close' to being Metzler in some sense. Distance problems arise frequently in the control literature. For example, the \emph{structured singular value} $\mu$, a cornerstone of modern robust analysis, is a measure of the distance to instability when considering structured perturbations~\cite{PacD93,ColS16}. The classical Nehari problem measures the shortest distance in $\cL_{\infty}$ between a causal and an anti-causal transfer matrix, and arises in optimal model reduction~\cite{Glo84} and the classical Youla approach to $\cH_{\infty}$-synthesis~\cite[Ch. 17]{ZDG}. Of particular relevance to this paper is the \emph{real stability radius}~\cite{QiuBRDYD95} problem, which can be thought of as a dual to the problem considered here, where we wish to find the smallest perturbation to a stable square matrix that renders it unstable. 

The problem of finding the nearest stable matrix (with no Metzler constraints) to an unstable one is in general a difficult problem. One of the main reasons for this is that the set of stable matrices does not form a convex set \cite{OrbNV13}, and thus formulating an optimization problem over the set is not trivial. Secondly, the spectrum of a stable matrix depends in a complicated manner on the coefficients of the matrix that defines it. For both of these reasons analytic solutions to the problem appear to be out of reach for anything but trivial cases. In this work we use the theory of dissipative Hamiltonian systems \cite{Sch06} and build on the framework developed in \cite{MehMS16,GilS16} to provide a convex approximation of the set of stable Metzler matrices.

In Section~\ref{se:Intro} we introduce the problem we are trying to solve  and some now well-known system-theoretic properties of \emph{positive systems}. In Section~\ref{se:DHS} the concept of (stable) \emph{dissipative Hamiltonian systems} and some useful characterising properties are described, and then in Section~\ref{se:Res} we present the main results and two algorithms that solve the relevant optimization problem. In Section~\ref{se:eg} illustrative numerical examples are presented.
\subsection*{Notation}
The notation used in this paper is standard. Let $H$ be a symmetric square matrix, then $H$ is said to be positive semidefinite, denoted $H\succeq 0$, if $x^THx\ge 0$ for all $x\in \R^n$. Further, $H$ is positive definite, denoted $H\succ 0$, if $x^THx>0$ for all $x\neq 0$. The inequality $\ge$ acts element-wise on a (not necessarily square) matrix, i.e. $H\ge 0$ implies $H_{ij}\ge 0$ for all $i,j$. Given a convex set $\cC$, the operator that projects a vector $z$ onto $\cC$ is denoted by $\cP_{\cC}z$. Finally, $\mathbb{S}_+^n$ and $\mathbb{G}^n$ denote the set of $n\times n$ positive semidefinite and skew-symmetric matrices; their dimension will be clear from context and so is omitted from the notation. 
\subsection{Problem Setup}
Assume that we are given a matrix $A\in \R^{n\times n}$, and we are interested in computing the distance to the `nearest' Metzler matrix, where a Metzler matrix is any square $n\times n$ matrix with non-negative entries on the off-diagonal elements, and the set of $n\times n$ Metzler matrices is denoted by $\cM^n$. We will restrict our attention to a search for the nearest \emph{asymptotically stable} Metzler matrix. Denote by $\cS^{n}$ the set of all asymptotically stable matrices of dimension $n$, i.e. $A\in \cS^n$ iff $\Re(\lambda_i)<0$ for $i=1,\hdots,n$. Note that this is an open set. Denote the set of asymptotically stable Metzler matrices of dimension $n$ by $\cM\cS^n$. The following theorem from \cite{Ran15} summarises some key Lyapunov stability results for systems of the form
\begin{equation*}
\dot{x}(t) = Ax(t), \quad A\in \cM^n.
\end{equation*}
\begin{thm}\label{thm:Ran}
Given a Metzler matrix $A\in \cM^n$, the following statements are equivalent:
\begin{enumerate}
\item $A$ is Hurwitz, i.e. $A\in \cM\cS^n$.
\item There exists a vector $\zeta \in \R^n$ such that $A\zeta<0$ with $\zeta>0$.
\item There exists a vector $z\in \R^n$ such that $z^TA<0$ with $z>0$.
\item There exists a diagonal matrix $P\succ 0$ such that $A^TP+PA \prec 0$. One such $P$ is given by $P_{ii}=\frac{z_i}{\zeta_i}$ for $i=1,\hdots,n$.
\item Lyapunov functions corresponding to statements 2) and 3) are given by $V(x)=\max_i~x_i/\zeta_i $ and $V(x)=z^Tx$ respectively.
\end{enumerate}
\end{thm}
From Theorem~\ref{thm:Ran}, the computational advantages of dealing with Metzler matrices should be clear -- the most prominent being that Lyapunov stability analysis can be carried out by solving an LP instead of a more intractable SDP. Thus, a reasonable question to ask is: how large is the smallest perturbation to a matrix that results in a  Metzler matrix?

A question which is broader than the one we seek to solve, but which we will build on, is to find the nearest stable matrix to a given unstable one. In posing this question the notion of \emph{distance} must be defined. In this setting the Frobenius-norm seems natural as it is a measure directly related to the matrices' coefficients. The Frobenius norm of a matrix $Z\in \R^{m\times n}$ is defined as
\begin{equation*}
\|Z\|_F = \sqrt{\trace (Z^TZ)} = \sqrt{\sum_{i=1}^m\sum_{j=1}^n |Z_{ij} |^2}.
\end{equation*}
The general problem can then be stated as follows: Given a matrix $A\in \R^{n\times n}$ where $A\notin \cS^n$, solve
\begin{equation}\label{eq:min_dist}
\inf_{X\in \cS^n}~\|A-X\|_F^2
\end{equation}
where  the optimal decision variable is denoted by $X^{\star}$. Note that we have to search for an infimum rather than a minimum since the set of stable matrices is an open set, owing to the fact that eigenvalues depend continuously on the matrix elements. Furthermore, the set of stable matrices does not form a convex set, at least not in the space of coefficients. 

An equivalent way of stating problem~\eqref{eq:min_dist} is based on a Lyapunov stability argument. The matrix $A$ is asymptotically stable if and only if there exists a $P\succ 0$ such that the Lyapunov operator 
\begin{equation*}
\cL_A(P)\triangleq AP+PA^T
\end{equation*}
satisfies $\cL_A(P)\prec 0$ (and stable if $\cL_A(P)\preceq 0$). We can now restate~\eqref{eq:min_dist} as
\begin{eqnarray}
\inf_{P\succ 0, ~X\in\cS^n }&\|A-X\|_F^2 \label{eq:dist_cons} \\
\text{s.t.}& \cL_X(P)\prec 0 \nonumber
\end{eqnarray}
where it can clearly be seen that~\eqref{eq:dist_cons} is not convex in $X$ and $P$. In \cite{OrbNV13} Orbandexivry, Nesterov, and Van Dooren use the theory of self-concordant barrier functions and the Dikin ellipsoid to compute a sequence of improving suboptimal approximations of $\Xstar$. 

We now state the specific problem we would like to solve: Given a matrix $A\in \R^{n\times n}$ (not necessarily belonging to $\cS^n$ or $\cM^n$), compute the distance in terms of the Frobenius-norm of the coefficients to the nearest \emph{stable} Metzler matrix, i.e. solve
\begin{align*}
 \inf_{X\in \cM\cS^n}~\|A-X\|_F^2&    \\ =\inf_{P\succ 0,X\in\cM\cS^n}&\|A-X\|_F^2\quad \text{s.t.} ~\cL_X(P)\prec 0 \\
 = \inf_{P\succ 0,X\in\cM^n}&\|A-X\|_F^2\quad \text{s.t.} ~\cL_X(P) \prec 0.
\end{align*}
Finally, in light of Theorem~\ref{thm:Ran}, the problem can be further reduced to 
\begin{equation}\label{eq:Metz_d_opt}
p^{\star}:=\inf_{\zeta >0,~ X\in \cM^n}~\|A-X\|_F^2 \quad \text{s.t.} ~X\zeta <0.
\end{equation}
To the best of our knowledge, this question has not been addressed in the literature. The contribution of this paper is to describe an iterative algorithm that comprises two convex optimization subproblems whose solution provides an upper bound to this non-convex problem. 

\section{Dissipative Hamiltonian Systems}\label{se:DHS}
Consider the linear dynamical system
\begin{equation}\label{eq:LTI}
\dot{x}(t) = Ax(t), \quad x(0)=x_0\in \R^n,
\end{equation}
where in this case there is no assumption on the structure of $A$. From \cite{GilS16} we say that the system~\eqref{eq:LTI} is said to be a \emph{dissipative Hamiltonian} (DH) system if and only if it can be expressed as 
\begin{equation}\label{eq:DH}
\dot{x}(t) = (J-R)Qx(t), 
\end{equation}
where $Q=Q^T \succ 0$, $R=R^T \succeq 0$, and $J=-J^T$ are fixed matrices of appropriate dimension. The physical interpretation is as follows: the Hamiltonian function $x^TQx$ describes the energy of the system, $R$ describes the energy dissipation or loss of the system, and $J$ describes the energy flux between storage elements in the system.  The following lemma from~\cite{GilS16} will prove useful in the sequel.
\begin{lem}\label{lem:DHstable}
If the eigenvalues of $A$ lie in the closed left half of the complex plane and all eigenvalues on the $j\omega$-axis are semisimple (in which case $\cL_A(P)\preceq 0$), then $A$ has a DH representation.
\end{lem}

From Lemma~\ref{lem:DHstable} we now have an equivalent characterisation of stable matrices in terms of the $n\times n$ matrices $J,Q,$ an $R$:
\begin{equation*}
\cS^n := \left\{ (J-R)Q~|~ J=-J^T, Q\succ 0,R\succeq 0 \right\}.
\end{equation*}
Note that the triple $\{J,Q,R\}$ that satisfies the above constraints is neither open or closed due to the mixture of definite and semidefinite constraints. However,  the above formulation allows us to restate the problem of finding the nearest stable matrix (problem~\eqref{eq:min_dist}) as follows: 
\begin{align}\label{eq:opt_reform}
\inf_{X\in \cS^n}~\|A-X\|_F^2 \quad &= \inf_{J=-J^T,R\succeq 0, Q\succ 0}\|A-(J-R)Q\|_F^2 \nonumber \\ &= \inf_{J=-J^T,R\succeq 0, Q\succeq 0}\|A-(J-R)Q\|_F^2,
\end{align}
where the final inequality is from Theorem 1 in~\cite{GilS16}.  Moreover, the authors in~\cite{GilS16} describe a selection of algorithms to solve the optimization problem on the right hand side of ~\eqref{eq:opt_reform}. An alternative approach that doesn't use the DH framework is provided in~\cite{OrbNV13}. Finally, we will make use of the following lemma from~\cite{GilS16} that characterizes imaginary axis eigenvalues: 
\begin{lem}\label{lem:DHeigs}
Let $J,Q,R\in \R^{n\times n}$ be such that $J=-J^T, Q\succeq 0,R\succeq 0$, then:
\begin{enumerate}
\item The spectrum of $(J-R)Q$ lies in the closed left half of the complex plane. Furthermore, $(J-R)Q$ has an eigenvalue on the $j\omega$-axis if and only if $RQx=0$ for some eigenvector $x$ of $JQ$.
\item All non-zero purely imaginary eigenvalues of $(J-R)Q$ are semisimple.
\end{enumerate}
\end{lem}
Note that lemma~\ref{lem:DHeigs} allows for $Q$ to be rank deficient. For notational convenience, the triple $\{J,Q,R\}$ satisfying the DH constraints will be denoted by
\begin{equation*}
\cD := \{ J,Q,R \in \R^{n \times n} ~|~ J=-J^T, Q\succeq 0, R\succeq 0\}.
\end{equation*}

\section{Results}\label{se:Res}
In this section an iterative method is presented that computes sub-optimal solutions to problem~\eqref{eq:opt_reform} in the DH setting, subject to the additional constraint that $X$ is Metzler. 

\subsection{Block Coordinate Descenct Algorithm}
We now focus on solving~\eqref{eq:opt_reform} subject to the constraints that $X=(J-R)Q\in \cM\cS^n$ with $\{J, Q, R \}\in \cD$. This results in the following optimization problem:
\begin{eqnarray}
 \inf_{J=-J^T,R\succeq 0, Q\succeq 0}&\|A-(J-R)Q\|_F^2 \label{eq:opt}\\
 \text{s.t.}& (J-R)Q\in \cM\cS^n \nonumber.
\end{eqnarray}
To begin with we will focus on the stable Metzler constraint. Clearly this is not a convex constraint as it requires finding matrices $J,R,Q,P$ such that 
\begin{align*}
[(J-R)Q]P+P[(J-R)Q]^T &\prec 0,\\
P&\succ 0,\\
\{J,Q,R\} &\in \cD.
\end{align*}
Furthermore, the DH framework is stated in terms of a stability result, not an asymptotic stability result. Ideally, we would like to avoid matrices with semisimple $j\omega$-eigenvalues. The following result addresses the asymptotic stability issue and uses diagonal stability in order to provide additional constraints to obtain a Metzler matrix. A block descent algorithm will decouple the problem into two convex subproblems that provide an approximate solution to~\eqref{eq:opt}.
\begin{prop}\label{prop:ASM}
Let $X=(J-R)Q$ be a stable matrix with $\{J,Q,R\}\in \cD$. If, additionally, $R\succ 0$, then $X$ is asymptotically stable. Furthermore, if $[J-R]_{ij}\ge 0$ for all $i,j\neq 0$ and $Q$ is diagonal, then $X$  is Metzler.
\end{prop}
\begin{proof}
By Lemma~\ref{lem:DHstable}, $X$ being stable implies it has a DH representation. Additionally, $X$ stable satisfies $\cL_X(P)\preceq0$, for $P\succ 0$. Now set $J=\frac{XP-PX^T}{2}$ , $R = \frac{-XP -PX^T}{2}$, $Q=P^{-1}$ and we arrive at one DH representation of $X$. From 1) in Lemma~\ref{lem:DHeigs}, for there to be no eigenvalues on the $j\omega$-axis we require $RQx\neq 0$ for any $x$ (note, this is a stronger requirement than in Lemma~\ref{lem:DHeigs} ). By Sylvester's inequality (Lemma~\ref{lem:S} below with $m=n$) we have that  $RQ$ is full rank. Moreover, $R\succ 0$ implies $\cL_X(P)\prec0$, and thus we have asymptotic stability. Imposing the constraints on $(J-R)$ as in the theorem and the diagonal consraint on $Q$ it is clear that the elements on the off-diagonal of $(J-R)Q$ will be non-negative thus we have a Metzler matrix. As the resulting matrix is Metzler, the diagonal constraint on $Q$ imposes no additional conservatism.
\end{proof}

\begin{lem}[Sylvester's Inequality]\label{lem:S} Let $A\in \R^{m\times n}$ and $B\in \R^{n\times p}$, then 
$\textbf{rank}(A) + \textbf{rank}(B) -n \le \textbf{rank}(AB).$
\end{lem}

Based on Proposition~\ref{prop:ASM} the following non-convex optimization problem will provide a solution to the nearest stable Metzler matrix problem described by~\eqref{eq:opt}:
\begin{align}
\inf_{J=-J^T,R\succ 0, Q\succeq 0}\quad &\|A-(J-R)Q\|_F^2  \nonumber \\
 \text{s.t.}\quad& Q_{ij}=0 \quad \forall i\neq j    \label{eq:relaxed}\\
 & [J-R]_{ij}\ge 0 \quad \forall i\neq j \nonumber.
\end{align}
Ignoring the constraints for the moment, it should be noted that restricting attention to the case where $Q\succ 0$ and right-multiplying the expression inside the norm operator by $Q^{-1}$ will convexify the problem. Unfortunately, it will also push towards a trivial solution with the smallest allowable spectrum of $Q$. Instead we will pursue a \emph{block coordinate descent} method \cite{Nes12} that freezes a subset of the variables, rendering the resultant problem convex, and then optimizes over the remaining subset with the previously computed variables held constant. From~\eqref{eq:relaxed}, a natural partition to choose is $\{J,R\}$ and $\{Q\}$. Algorithm~\ref{alg:NM} below provides the details of such an implementation.

\begin{algorithm}
\caption{Find Nearest Metzler: Coordinate Descent}\label{alg:NM}
\begin{algorithmic}[1]
\Require $J=-J^T, Q\succeq 0, R \succ 0$ and $maxiter$. 
\State $iter\gets 1$
\While{$iter \le maxiter$}
	\State Fix $Q$ and solve 
	\begin{align} \label{eq:optJR}
		\inf_{J'=-J'^T,R'\succ 0}\quad &\|A-(J'-R')Q\|_F^2 \\
		\text{s.t.}\quad& [J'-R']_{ij}\ge 0 \quad \forall i\neq j \nonumber
	\end{align}
	\State $J\gets J'$, $R\gets R'$ 
	\State Fix $J,R$ and solve 
	\begin{align} 
		\inf_{Q'\succeq 0}\quad &\|A-(J-R)Q'\|_F^2 \label{eq:optQ}\\
		\text{s.t.}\quad & Q'_{ij}=0 \text{~if~} i\neq j \nonumber
	\end{align}
	\State $Q\gets Q'$
	\State $iter \gets iter + 1$ 
	\EndWhile
\Ensure{$J,R,Q$ such that $X=(J-R)Q$ is asymptotically stable and Metzler.}
\end{algorithmic}
\end{algorithm}
The two subproblems~\eqref{eq:optJR} and~\eqref{eq:optQ} are both convex and have nice structural properties, and Algorithm~\ref{alg:NM} is guaranteed to converge as it is a two-block problem\cite{GriS00}. For the non-Metzler version of this problem, the work in~\cite{GilS16} goes even further than the two-block approach given above: they are able to implement a fast projected gradient scheme which outperforms the two-block descent method. This is possible due to the fact that analytic projections onto $\cS_+$, as well as the set of skew-symmetric matrices, are easily derivable. For the case of Algorithm~\ref{alg:NM} it is not immediately obvious how to derive the 
projections to take into account the element-wise positivity constraints on $J-R$ and the diagonal constraint on $Q$. However, in Sections~\ref{sec:JR} and \ref{sec:Q} respectively we will exploit a computational relaxation for~\eqref{eq:optJR} and the data structure in the case of subproblem~\eqref{eq:optQ}.
\subsection{Optimizing over $J,R$}\label{sec:JR}
The authors in~\cite{GilS16} note that the subproblems in Algorithm~\ref{alg:NM} are implementable by first-order methods and thus may scale well with problem size. In this work we will solve subproblem~\eqref{eq:optJR}  in its natural form and also look at improving scaling by relaxing the semidefinite constraint to a more simple cone constraint.

The convex optimization subproblem~\eqref{eq:optJR} is a semidefinite optimization problem (SDP) -- or at least is trivially convertible to one. Such problems have been shown to have a polynomial time complexity (see ~\cite{VanB96} for a review of the subject and a description of their numerical implementation). Dropping the element-wise inequality from~\eqref{eq:optJR}, the methods of~\cite{GilS16} reduce the problem to a projected gradient descent where the projections $\cP_{\mathbb{G}}$ and $\cP_{\mathbb{S}_+}$ are defined by
\begin{equation*}
\cP_{\mathbb{G}}(X) = \frac{X-X^T}{2}~ \text{and~} \cP_{\mathbb{S}_+}(X) = UDU^T,
\end{equation*}
where $D = \text{diag}(\max(0,\lambda_1),\hdots, \max(0,\lambda_n))$, $U$ is orthonormal, and $\lambda_i$ is the $i^{\text{th}}$ eigenvalue of $X$. Unfortunately, the constraint $[J-R]\ge 0$ on the off-diagonal elements of $J-R$ prohibits the use of these projections. As a result we will either solve~\eqref{eq:optJR} as it is specified, or, at the cost of finding suboptimal solutions, introduce two relaxations of the problem that have more attractive computational tractability.

\begin{defn}
A symmetric matrix $F\in \R^{n\times n}$ is diagonally dominant if $F_{ii}\ge \sum_{j\neq i}|F_{ij}|$ for $i=1,\hdots,n$. Furthermore, if there exists a positive definite diagonal matrix $D$ such that $DFD$ if diagonally dominant then we say that $F$ is scaled diagonally dominant. 
\end{defn}
Denoting the set of $n\times n$ diagonally dominant and scaled diagonally dominant matrices by $DD^n$ and $sDD^n$ respectively, it is straightforward to show that the cone of positive definite matrices contains $DD^n$ and $sDD^n$ and that $DD^n \subseteq sDD^n$. What is particularly useful about this is that the constraints which enforce diagonal dominance are linear, whilst the $sDD$ constraints can be enforced via a second-order cone constraint~\cite{AhmH15,SOCP}. To see this, note that $F\in sDD^n$ if and only if it admits a decomposition $F=\sum_{i<j}M^{ij}$, where $M^{ij}$ is a symmetric $n\times n$ matrix with zeros everywhere apart from at $M_{ii}, M_{ij}$ and $M_{ij}=M_{ji}$, which together make the matrix $\left[ \begin{array}{ll} M_{ii} & M_{ij}\\ M_{ij} & M_{jj} \end{array}\right]$ positive semidefinite. Such constraints are called rotated quadratic cone constraints and are imposed as 
\begin{equation*}
\left\| \left(\begin{array}{c} 2M_{ij} \\ M _{ii}- M_{jj}  \end{array} \right)  \right\| \le M_{ii} + M_{jj}, \quad M_{ii}\ge 0.
\end{equation*} 
For the case of $DD^n$ matrices, let $M=M^T$, then $M\in DD^n$ and consequently $M \succ 0$ if the linear constraints 
\begin{align*}
M_{ii} > 0, \quad M_{ii} > \sum_{j\neq i}^n y_{ij}, \quad i=1,\hdots,n,\\
-y_{ij}\le M_{ij}\le y_{ij}, \text{ with } y_{ij}=y_{ij} ~\forall i\neq j
\end{align*} 
are feasible. Applying either of these relaxations to~\eqref{eq:optJR} gives
\begin{align*}
\inf_{J=-J^T}\quad &\|A-(J-R)Q\|_F^2 \\
\text{s.t.} \quad& [J-R]_{ij}\ge 0 \quad \forall i\neq j \\
\quad & R \in \cD\cD^n
\end{align*}
where $\cD\cD^n$ is chosen to be either $DD^n$ or $sDD^n$.

\subsection{Optimizing over $Q$}\label{sec:Q}
The optimization subproblem~\eqref{eq:optQ} is a quadratic program (QP) with linear constraints \cite[Ch. 4.4]{CONVEX}. The general form of such a problem is
\begin{align}\label{eq:optQP}
\text{minimize} \quad & \frac{1}{2}x^TPx + q_0^Tx +r_0 \\
\text{s.t.}\quad &  q_i^Tx +r_i \le 0, \quad i=1,\hdots,m \nonumber,
\end{align}
where $x\in \R^n$ is the decision vector and $P\in \R^{n\times n}, q_i \in \R^n$, and $r_i \in \R$ with $P\succeq 0$ are given. In the case of optimization~\eqref{eq:optQ} there is significant structure to be taken advantage of. Specifically the objective function matrices take the form
\begin{equation*}
P = \left[\begin{array}{ccc} \sum_{i=1}^nB_{i1}^2 &  0 &0 \\ 0 & \ddots & 0 \\ 0 & 0& \sum_{i=1}^nB_{in}^2 \end{array} \right], 
\end{equation*}
\begin{equation*}
q_0 = -2\left[\begin{array}{c} \sum_{i=1}^n A_{i1}B_{i1}\\ \vdots \\ \sum_{i=1}^n A_{in}B_{in} \end{array}\right], \quad r_0 = \sum_{i,j}^nA_{ij}^2=\|A\|^2_F,
\end{equation*}
where $B:=J-R$. The constraints of~\eqref{eq:optQ} map to the constraints of~\eqref{eq:optQP} via $m=1$, $q_i^T = -\textbf{1}_{1\times n}$ and $r_i = \epsilon$. By default $\epsilon = 0$, but if one wanted to fix positive definiteness of $Q$ it could be set to some arbitrarily small positive scalar. Given the structure described, we now propose a simple first-order method to solve~\eqref{eq:optQ} using the \emph{Alternating Direction Method of Multipliers} (ADMM)\cite{ADMM}. Problem~\eqref{eq:optQ} can be regarded as minimizing a convex function subject to the constraint that $x\in \cC$ where $\cC$ is a convex set. In this case $\cC = \R_+^n$. The ADMM approach solves such a problem by introducing a second decision vector $z\in \R^n$,
 and solves
 \begin{align*}
 \text{minimize} \quad & f(x) + g(z)\\
 \text{s.t} \quad & x = z,
 \end{align*}
where $g$ is the indicator function for $\cC$. The augmented Lagrangian function with penalty factor $\rho >1$ is defined as
\begin{equation*}
L_{\rho}(x,z,u) = f(x)+g(z) + \frac{\rho}{2}\|x-z+u\|_2^2.
\end{equation*}
The simplest form of ADMM iterates over minimizing $L_{\rho}$ for fixed $x$, then fixed $z$. The scaled ADMM (c.f.~\cite[p.15]{ADMM}) formulation iterates as follows:
\begin{align*}
x^{k+1} &: = \argmin_x ~ \left[ f(x) +  \frac{\rho}{2}\|x-z^k+u^k\|_2^2 \right]\\
z^{k+1} & := \cP_{\cC}(x^{k+1}+u^k)\\
u^{k+1} & := u^k + x^{k+1} - z^{k+1}.
\end{align*}
For subproblem~\eqref{eq:optQ} the $x$ update has an analytic formula derivable from the KKT-conditions, and the projection in the $z$-update is particularly simple. The resulting iterates are 
\begin{align*}
x^{k+1} &= (P+\rho I)^{-1}q_0 - \rho (P+\rho I)^{-1}(z^k - u^k) \\
z^{k+1} & = (x^{k+1}+u^k)_+. 
\end{align*} 
Note that $P+\rho I$ is a diagonal matrix, hence its inversion simply involves $n$ scalar division operations. Additionally, the matrix $(P+\rho I)^{-1}$ can be computed offline \emph{a priori}. Under the assumptions that i) the extended functions $f,g:\R^n \rightarrow \R \cup \{+\infty\}$ are closed proper and convex, and ii)  there exists a saddle point of the un-augmented Lagrangian $L_0$, then the ADMM iterates will converge and the residual $x^k+z^k \rightarrow 0$ as $k \rightarrow \infty$ (c.f. \cite[p. 17]{ADMM}).

\subsection{Duality}
In order to assess the performance of the algorithm presented above we will compare its solutions to the global lower bound computed by solving the dual to problem~\eqref{eq:Metz_d_opt}. The derivation and analysis of the dual are the focus of ongoing work, however for completeness it is given below:
\begin{align}\label{eq:dual}
\sup_{\lambda \in \R^{n\times n}} \quad & -\frac{1}{4}\|\Lambda\|_F^2 - \langle \Lambda,A\rangle \nonumber \\
\text{s.t.} & \quad \Lambda_{ii} = 0, \quad i=1,\hdots, n\\
& \quad\Lambda_{ij} \ge  \quad \forall i\neq j \nonumber
\end{align}
and the optimal value is denoted by $d^{\star}$.
\begin{prop}
There is a non-zero duality gap for problems \eqref{eq:Metz_d_opt} and \eqref{eq:dual}, i.e. $p^{\star}>d^{\star}$.
\end{prop}

\section{Numerical Examples}\label{se:eg}
In the first example we consider finding the nearest stable Metzler matrix to that of a stable $5\times 5$ matrix generated using Matlab's \texttt{rss} command. The matrix generated was 
\begin{equation*}
A_1 = \left[ \begin{array}{rrrrr} -1.733  &  1.295 &  -0.497  &  0.765  &  0.763  \\
    0.481 &  -1.472 &  -0.945  &  1.381 &   0.146 \\
    0.680   & 0.326 &  -1.392 &  -0.536 &   1.957 \\
   -1.442 &  -1.127 &  -0.355 &  -1.079   & 1.375 \\
    0.566  &  0.008  &  1.849  &  1.607  & -6.299 \end{array}\right]
\end{equation*}
given to 3dp. The SDP solver chosen was SeDuMi~\cite{SEDUMI}, which was used in conjunction with the modelling tool YALMIP~\cite{YALMIP}. Applying Algorithm~\ref{alg:NM} produces the Metzler matrix
\begin{equation*}
X _1= \left[ \begin{array}{rrrrr} -1.807   & 1.234 &   0 &   0.717 &   0.720 \\
    0.402 &  -1.536  &  0  &  1.330  &  0.101 \\
    0.582  &  0.245  & -1.507  &  0  &  1.901 \\
    0  &  0  &  0  & -1.216  &  1.252 \\
    0.487  &  0  &  1.756   & 1.556  & -6.344 \\\end{array}\right],
\end{equation*}
where $\|A_1-X_1\|_F^2 = 5.019$, $d^{\star} = 1.223$, and $X_1$ is asymptotically stable with eigenvalues 
\begin{equation*}
\lambda(X_1)=\{-7.281, -0.005, -2.433, -1.346 \pm 0.578 j\}.
\end{equation*}
Next we consider the problem of finding the nearest Metzler matrix to that of an unstable matrix $A_2$ where
\begin{equation*}
A_2 = \left[\begin{array}{rrrrr} 0.647     & 0.172      & -0.749      &  0.728    & 0.717 \\
   -0.354  &  -0.062   & -0.936   &-0.773   & -0.778\\
    0.046  &   1.199   & -1.269    & 0.837    & 0.316\\
   -0.793  &  0.802   & 0.498   & -1.128    & 1.407\\
   -1.551  &  1.053  &  2.789   & -1.425    & 0.401 \end{array} \right].
\end{equation*}
First Algorithm 1 is run using the default constraint of $Q\succ 0$. The result is the stable Metzler matrix
\begin{equation*}
X_2 = \left[\begin{array}{rrrrr}-0.059 &   0.170 &   0.003  &  0.665 &   0.6552 \\
    0  & -0.173  &  0.03   & 0 &   0\\
    0   & 1.180  & -1.316   & 0.008  &  0\\
    0    &0.801  &  0.495  & -1.178  &  1.357\\
    0   & 1.040&    2.756   & 0 & -0.183\\ \end{array} \right],
\end{equation*}
where $\|A_2-X_2\|_F^2 = 9.485$, $d^{\star}=1.957$, and $X_2$ has eigenvalues 
\begin{equation*}
\lambda(X_2)=\{ -0.059, -0.175, -0.155, -1.26\pm 0.139j \}.
\end{equation*}

\begin{figure}[t]
  \centering
  \includegraphics[scale =0.325]{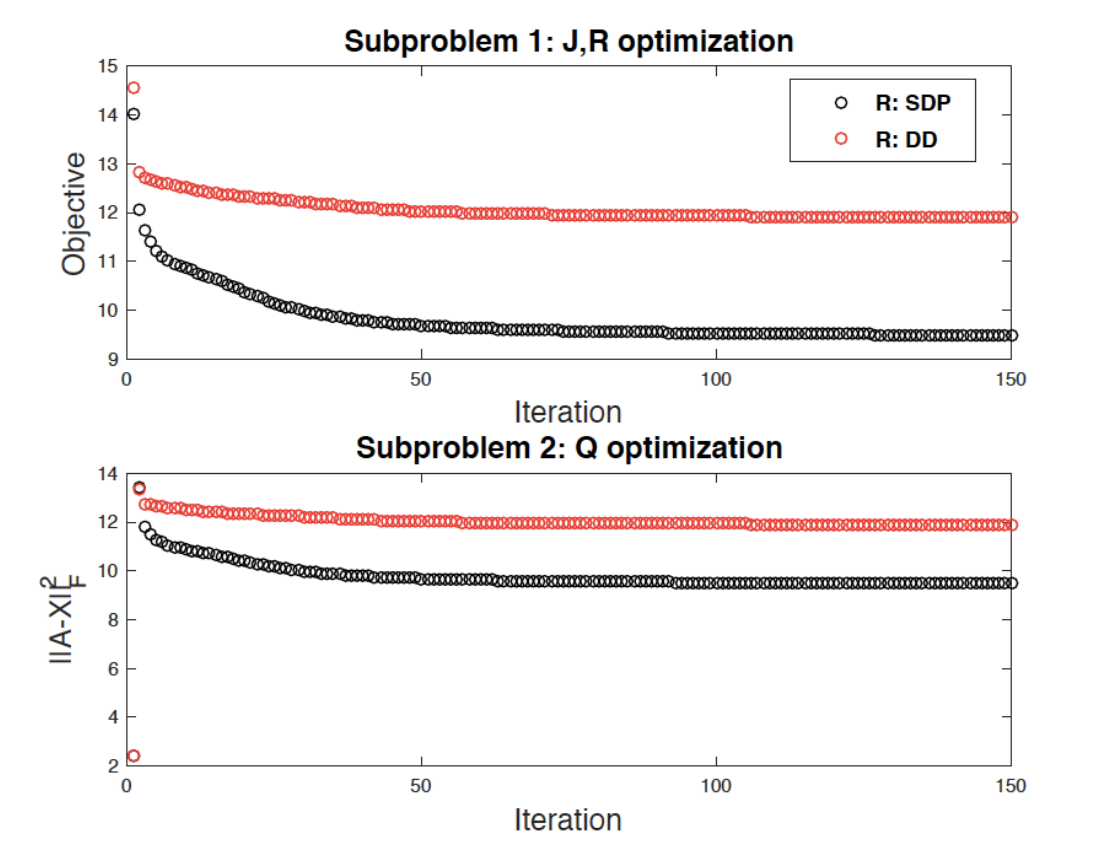} 
  \caption{The objective values for the subproblems in Algorithm~\ref{alg:NM}. Top: Subproblem~\eqref{eq:optJR}, optimizing over $J$ and $R$. Iterates in black correspond to $J$ being a positive semidefinte constraint. Iterates in red correspond to a relaxed diagonally dominant constraint. Bottom: $Q$ solved as a QP corresponding to a SDP constraint on $J$ (black) and a diagonally dominant constraint (red). }\label{fig:converge}
\end{figure}

A plot of the convergence for both subproblems is shown in black in Figure~\ref{fig:converge}. Note that the overall objective function is equal to the objective of the second subproblem, in this case minimizing $Q$. In this example the convergence of the SDP step is particularly slow.

 Next we take the same matrix $A_2$ but constrain $R$ to be diagonally dominant. With this constraint we obtain$\|A_2-X_3\|_F^2 = 11.903>\|A_2-X_2\|_F^2$ as expected. The resulting Metzler matrix $X_3$ is given below
\begin{equation*}
X_3 = \left[ \begin{array}{rrrrr}   -0.011   & 0 &   0  &  0.713  &  0.216 \\
    0  & -0.895   & 0  &  0   & 0\\
    0   & 0.769  & -1.542  &  0 &  0\\
    0   & 0.385   & 0.361  & -1.145  &  0.901\\
    0    & 0.636 &   2.645   & 0  & -0.580\\
\end{array}\right],
\end{equation*}
with $\lambda(X_3)= \{-0.011, -0.895, -0.580, -1.542, -1.145\}$. The convergence of the iterates is shown in red in Figure~\ref{fig:converge}.

The algorithms described were implemented on test cases of full matrices up to order $n=35$, beyond which the number of iterations required to converge became prohibitive. For sparse matrices we were able to solve larger but not significantly larger problems. The main reason for this is that we are not taking full advantage of sparsity with the solver; recent methods such as those proposed in~\cite{ZheFPGW16} may improve this. In addition, the penalty function $\rho$ in the QP subproblem was not optimized, but further performance may be achievable by implementing the results from~\cite{EuhTSJ15}.

\section{Conclusion}
An iterative method for obtaining suboptimal solutions to the nearest Metzler matrix problem has been presented. The proposed algorithm decomposes the problem into two convex problems: an SDP, and a convex QP for which a first-order optimization scheme with analytic iterates is described. In the case of the SDP, a simple relaxation to an LP or SOCP was described that comes with a suboptimality trade-off. 

An alternative approach to solving this problem would be to adapt the method from~\cite{OrbNV13} which uses the Dikin ellipsoid method to directly solve $\inf_{X,P} \|A-X\|_F^2$ s.t. $\cL_X(P) \prec 0, P\succ 0$ and adapt it to take advantage of the diagonal stability result of Theorem~\ref{thm:Ran}. As noted in~\cite{SooA16}, there is a much richer set of matrices that admit diagonal Lyapunov functions than Metzler matrices alone. It would be interesting to extend this work to take this fact into account. 

Current work is underway to construct a more useful dual problem and characterize the nature of the duality gap for this problem and the case where $X$ is no longer constrained to be Metzler.

\section{Acknowledgements}
I would like to thank Riley Murray at Caltech for many helpful discussions regarding the dual problem. The outcome of which is currently being written up.

\bibliographystyle{IEEEtran}
\bibliography{biblio}




\end{document}